\documentclass[11pt]{article}
\usepackage{latexsym,amssymb,amsmath,amsfonts,amsthm,graphicx,url}
\def\clearp{}

\newtheorem{theorem}{Theorem}
\newtheorem{lemma}{Lemma}[section]

\theoremstyle{definition}

\def\beq{ \begin{equation} }
\def\eeq{ \end{equation} }
\def\mn{\medskip\noindent}

\def\square{\vcenter{\vbox{\hrule height .4pt
  \hbox{\vrule width .4pt height 5pt \kern 5pt
        \vrule width .4pt} \hrule height .4pt}}}
\def\eopt{\hfill$\square$}

\def\ep{\epsilon}

\def\ZZ{\mathbb{Z}}
\def\LL{\mathbb{L}}

\def\ER{Erd\"os-R\'enyi\ }
\def\ERns{Erd\"os-R\'enyi}

\begin{document}

\title{Unusual properties of contact \\ processes on percolated graphs}
\author{Rick Durrett\thanks{The author was supported by NSF Grant DMS 2153429 from the probability program.}}

\date{\today}						

\maketitle

\begin{abstract}
In this paper we will consider the contact process in a very simple type of random environment that physicists call the random dilution model. We start with the contact process on a graph, here either $\ZZ^d$, a $d$-dimensional torus or an \ER graph, and then flip independent $(1-p)$ coins to delete edges, or delete vertices. Let $p^*$ be the threshold for percolation in the diluted graph. We will primarily be concerned with two phenomena. (i) The critical value for the contact process on the dliuted graph $\lambda_c(p)$ does not converge to $\infty$ as $p \downarrow p^*$. (ii) In contrast to the contact process on a homogeneous graph, the density of 1's starting from all sites occupied converges to 0 at a polynomial rate when $p<p^*$ (the ``Griffiths phase'') and like $c/(\log t)^a$ when $p=p^*$. 
\end{abstract}

\section{Introduction}

The state of the contact process on a graph $G$ at time is $\xi_t \subset G$, where the sites in $\xi_t$ are occupied (by particles) and the others are vacant. We will also write the state as $\xi_t: G \to \{0,1\}$ where 0 is vacant and 1 is occupied. Particles die at rate 1. For each unoriented edge $\{x,y\}$ in graph with $\xi_t(x)=1$ and $\xi_t(y)=0$ births from $x$ to $y$ occur at rate $\lambda$. In the bond verson of the dilution model, edges are indpendently deleted with probability $1-p$, and if not deleted have births across them at rate $\lambda$. In the site version, vertices are independently designated as {\bf active} with probability $p$ in which case they give birth at rate $\lambda$ across each adjacent edge. If a sites is not active, it is called {\bf inert} and does not give brth.

One can generalize the site version so that with probability $p$ sites give birth at rate $\lambda$ across all of their edges and with probability $(1-p)$ give birth across  edges at rate $r\lambda$ where $0<r<1$. Understanding the behavior of this process is difficult since the critical value no longer coincides with the onset of percolation. See Vojta and Dickison \cite{VojDic} and Vojta, Farquahr, and Mast \cite{VojFM}. Bramson, Durrett, and Swindle \cite{BDS} studied the one-dimensional contact process with constant birth rates and two death rates. They demonstrated the existence of an intermediate phase in which the contact process survives but does not grow linearly.
For more more on the one dimensional contact process in a random environment see Liggett \cite{Lig92}, and Newman and Volchan \cite{NewVol}.

Returning to our main subject, Noest \cite{Noe86} considered the case of oriented site percolation in $D+1$ dimensions. This process  takes place on the graph 
$$
{\cal L}^D =  \{ (x,n) : x \in \ZZ^D, n \ge 0,  x_1 + \cdots + x_D + n 
\hbox{ is even} \}.
$$
 with connections from $(x,n) \to (y,n+1)$ if $x$ and $y$ are nearest neighbors in $\ZZ^D$. The site $(y,n+1)$ is occupied with probability $c$ if at least one of its neighbors $(x,n)$ are occupied, otherwise it is vacant.

\begin{figure}[ht] 
  \centering
  \includegraphics[width=3.0in,keepaspectratio]{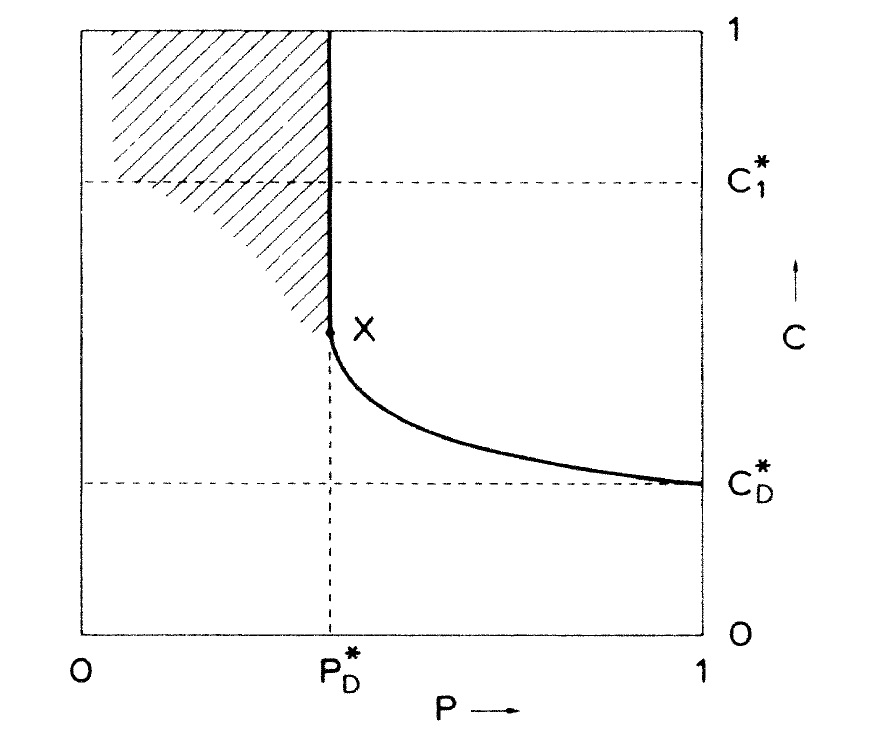}
  \caption{Phase diagram in $D>1$ for randomly diluted site percolation as a function of the probability $p$ that a site is retained and the birth probability $c$ in oriented site percolation.}
  \label{fig:Nfig}
\end{figure}

There is a probability $p^*$ so that when $p \le p^*$ the network of active sites does not form a percolating cluster. On such a network, the existence of a nontrivial stationary distribution for the process is not possible. Noest \cite{Noe86} has drawn a picture of the phase diagram for a generic random dilution model in $D>1$ dimensions that we reproduce in Figure \ref{fig:Nfig}.

A second, more recent, set of results in the physics literature concerns the  contact process on randomly diluted \ER graphs. See e.g.,  Mu\~noz, Juh\'asz,  Castellano, and \'Odor  \cite{MJCO}. Each site in the graph is independently assigned a birth rate, which is $\lambda$ with probability $p$, and $0$ with probability $1-p$.  Juh\'asz, \'Odor, Castellano and Mu\~noz \cite{JOCM} have described the phase diagram, which we have drawn in Figure \ref{fig:Jfig}.

\begin{figure}[ht] 
  \centering
  \includegraphics[width=3.5in,keepaspectratio]{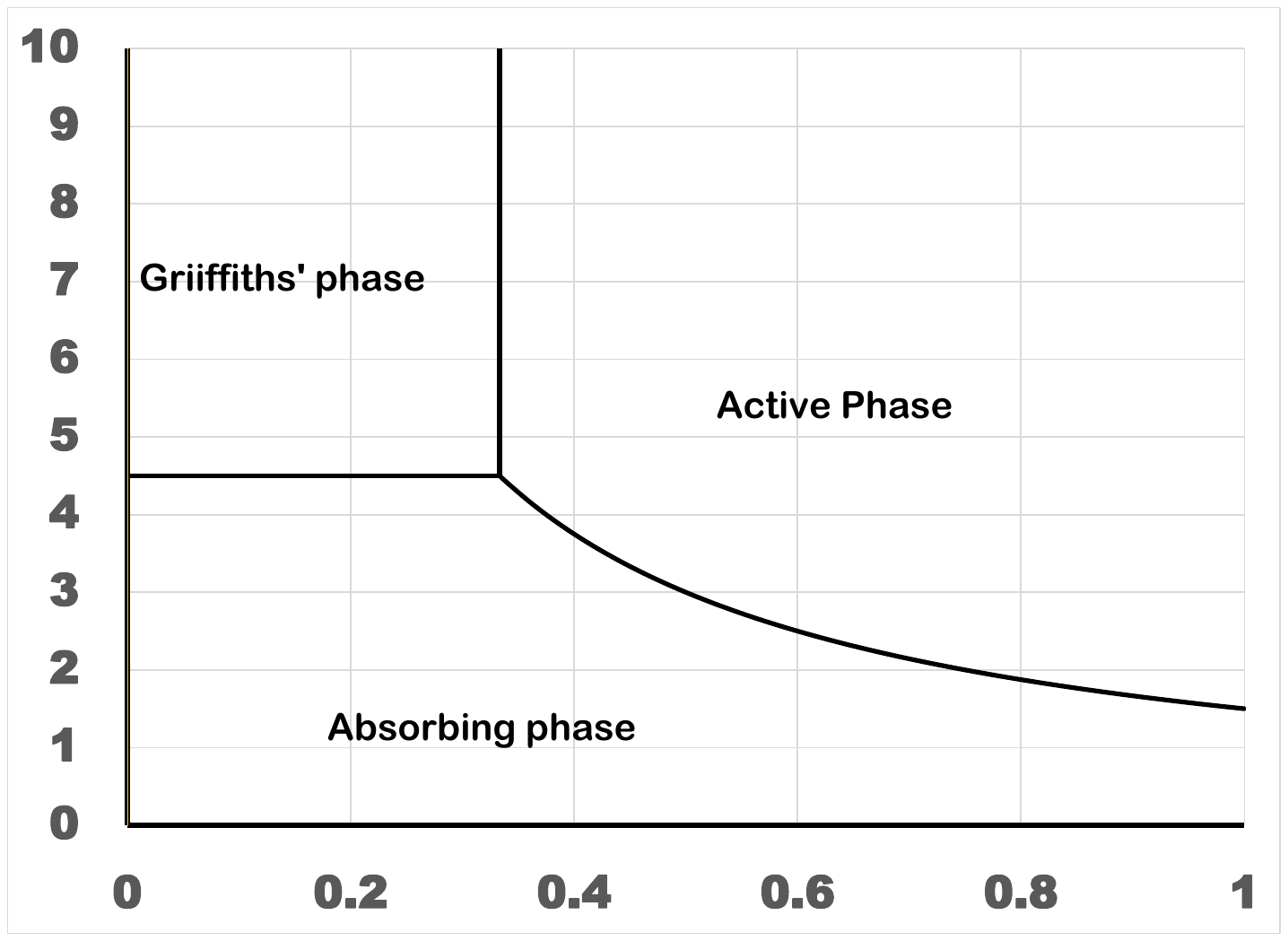}
  \caption{Phase diagram of the random dilution model on an \ER graph with mean degree $\mu = 3$, as a function of the fraction of active nodes $p$, plotted on the $x$-axis and birth rate $\lambda$ plotted on the $y$ axis. The percolation threshold is $p^*=1/3$.}
  \label{fig:Jfig}
\end{figure}

\noindent
{\bf Remark.} A special feature of the \ER graph is that if we start with an \ERns($N,\mu/N$) graph  and delete a fraction $1-p$  of the edges then we end up with  \ERns($N,p\mu/N$). If we delete a fraction $1-p$ of the vertices then we end up with a graph that is asymptotically \ERns($pN,\mu/N$). 

\medskip
Understanding the phase diagrams in the two figures is the main goal of this paper. We will consider five main features, but we will not have much to say about the last two.

\mn
{\bf 1. Supercritical behavior.} On tori in $D>1$ and on \ER graphs the critical value $\lambda_c(p)$ remains bounded as the fraction of active sites decreases to the critical value $p^*$. Intuitively, this occurs because  when $p>p^*$ there is not only a giant component but there is also a path of length $\ge \ep(p)N$ with $\ep(p)>)$, so as shown in Figure \ref{fig:Nfig} the multicritical point $X \le c^*_1$, the critical value in one dimension.  

\mn
{\bf 2. The Griffiths phase} is labeled in Figure \ref{fig:Jfig}. It is the striped region in the Figure \ref{fig:Nfig}. Griffiths' paper \cite{Gri69} concerned the randomly diluted Ising model and showed that the magnetization fails to be an analytic function of the external field $h$ when $h=0$ for a range of temperatures above the critical temperature (which is the subcritical phase of the Ising model).  In the case of the contact process (or oriented percolation) the phrase refers to the fact that in the subcritical region decay to the empty state occurs at a power law rate rather than the usual exponential rate on homogeneous graphs. This change can be explained by observing that in the subcritical regime clusters have a size distribution that has an exponential tail, so the largest cluster is of size $\ge b(p) \log N$ with $b(p)>0$ while the contact process survives for a time that grows exponentially in the size of the cluster. Thus if we start with all sites in state 1 then the system survives for time
$$
\ge \exp( a(\lambda) b(p) \log N) = N^{\kappa(p,\lambda)}.
$$

\mn
{\bf 3. Properties on the critical line, $p=p^*$,} has been studied by Moreira and Dickman \cite{MorDic} and their mirror image twins Dickman and Moreira \cite{DicMor}. One of their findings, which we can give a rigorous explanation for, is the fact that when $p=p^*$ the system survives for much longer than the polynomial amount of time in the homogeneous case.  The intuition is similar to the explanation of the Griffiths' phase in paragraph 2, but now the largest cluster sizes are $\ge b(p)n^\alpha$ for some $\alpha, b(p)>0$, which depends upon the model, so that the survival time is 
$$
\ge\exp(a(\lambda) b(p)n^\alpha).
$$

\mn
{\bf 4. Behavior near the critical curve $(p,\lambda_c(p))$, $p>p^*$.} It is expected that when $p$ is fixed and $\lambda$ decreases to $\lambda_c(p)$ the surivial probability tends to 0 like $(\lambda-\lambda_c(p))^{\beta(p)}$ where $\beta(p)$ is not constant, but varies continuously in $p$. Figure 3 graphs numerical results from \cite{MorDic}. The irregular shape of the graph of $\beta(p)$ indicates that there are substantial uncertainties in the estimated values.

\begin{figure}[ht] 
  \centering
  \includegraphics[width=5.0in,keepaspectratio]{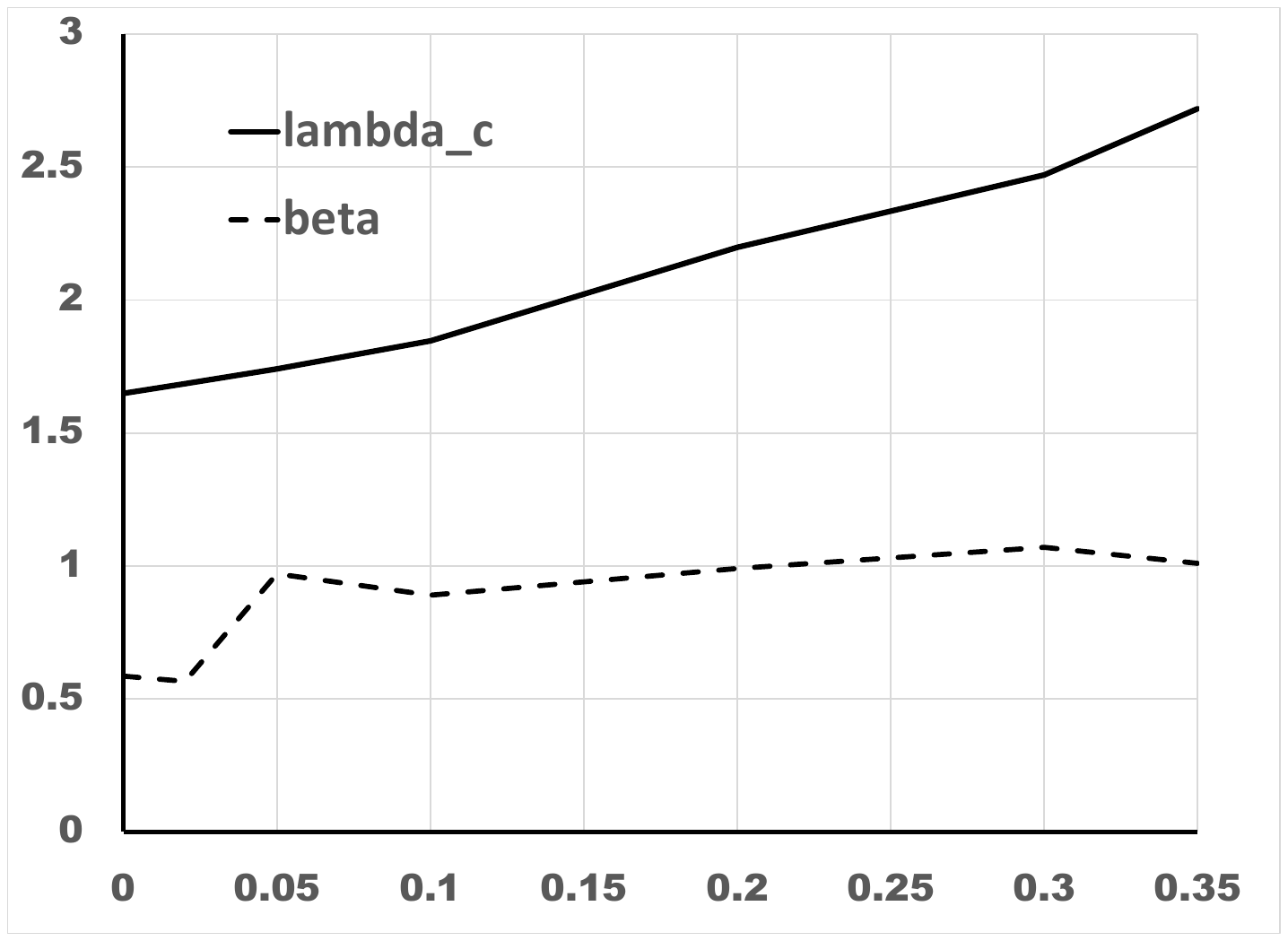}
  \caption{Two dimensional diluted contact process. Simulated critical values $\lambda_c(1-x)$ and critical exponents $\beta(1-x)$ for the survivial probability,  where $x$ is the fraction of deleted sites plotted on the $x$-axis. Here $\lambda$ is the total birth rate not the birth rate per edge.   }
  \label{fig:MDfig}
\end{figure}

\mn
{\bf 5. Critical exponents.} It is natural to ask if critical expoents are changed by randomness. A probabilist might think ``of course they are going to change,''  but physicists organize their models into universality classes in which the critical exponents are equali. For exanple, the DP universality class contains directed percolation, the contact processes, and a number of other processes with  local growth in which all 0's is an absoribing state. See Henrischen \cite{HH} for more details about this class of models.

There is a criterion due to A.B. Harris in a paper \cite{ABH} published in 1974, the year T.E. Harris defined the contact process in \cite{TEH74}. The condition is in terms of the spatail correlation length $L_{\perp}$ or more precisely in terms of the associated critical exponent
 $$
L_{\perp} \approx (\lambda-\lambda_c)^{-\nu_{\perp}}.
$$
For a precise definition of $L_\perp$ for the contact process see \cite{DST}.

Harris has a simple argument in terms of fluctutations, which shows that critical exponents are not changed if $D\nu_\perp> 2$. Using values of $\nu_\perp$ from sinultation he concluded that for the contact process they are changed if $d<4$. Chayes, Chayes, Fr\"ohlich, and Spencer \cite{CCFS} have proved a rigorous version of Harrs' criterion but that involves the notion of finite size scaling so we refer the reader to \cite{CCFS} for details..

\clearp

\section{Supercritical results}

Our first topic is the observation that despite that fact that on the graphs we consider there is no percolation at the critical value $p^*$,
$$
\lim_{p \downarrow p^*} \lambda_c(p) < \infty
$$
The limit exists by monotonicity.
 There is no supercritical phase for percolation in $D=1$, so we begin with 

\subsection{\ER graphs}

When the mean degree  is $\mu >1$ there is a giant component. Ajtai, Komlos, and Szemeredi \cite{Ajtai} were among the first to prove the fact that when $\mu>1$ not only is there is a connected component of size $\sim\ep(\mu)n$, but there is a path with length $\ge \eta(\mu)n$ where $\eta(\mu)>0$.  Using depth-first search (DFS), Krivelevich and Sudakov \cite{KriSud} have given a simple proof of this result and Enriquez, Faraud, and M\'enard \cite{EFM} have proved the following result with a very good constant. It precisely describes the limiting behavior of DFS for \ER graphs but DFS is only one way to generate a long path.

\begin{lemma} \label{longpER}
There is a function $\eta_{ER}: (1,\infty) \to (0,1)$ so that an \ER random graph with $N$ sites mean degree $\mu$ contains a path of length at least $\eta_{ER}(\mu)N$.
\end{lemma}

With the existence of a long parth established, we can get a lower bound on the survival time by using a result of Durrett and Schonmann \cite{DurSch}  

\begin{lemma} \label{d1surv}
The supercritical nearest neighbor contact process on $[1,N]$ starting from all sites occupied survives for time $\sigma_N$ where 
\beq
(1/N) \log \sigma_N\to \gamma_2(\lambda)>0  \qquad\hbox{in probability}.
\label{rstime}
\eeq
\end{lemma}

\noindent
The same result, with different constants, holds for oriented bond and site percolation in $1+1$ dimensions.

\begin{theorem} \label{lbsurvER}
Suppose $\nu = \mu p>1$, $\lambda > \lambda_c(\ZZ)$, and $\delta>0$. The  contact process on bond diluted \ERns$(N,\mu/N)$ started from all sites occupied survives for time 
$$
\sigma_N\ge \exp((1-\delta)\gamma_2(\lambda) \eta_{ER}(\nu) N)
\qquad\hbox{ for large $N$}.
$$ 
\end{theorem}

\subsection{Two dimensions}

We start with $Q^2_N =\ZZ^2 \cap [1,N^{1/2}]^2$ with edges to nearest neighbors that are also in the set and do bond dilution, where we keep edges with probability $p$. One important reason for deleting edges is that in two dimensional bond percolation we are able to exploit {\bf planar graph duality} between percolation on $Q^2_N$ and on a dual graph described in Section \ref{sec:Pf2d}. Using duality and some well-known facts about sponge crossings in two dimensional percolation, we are able to show that there is a positive constant $\eta_2(p)>0$ so that with high probability we have a path of length $\ge \eta_2 (p)N/\log N$. Using our result for the survival time of the one-dimensional contact process, Lemma \ref{d1surv}, now we have

\begin{theorem} \label{dens2D}
Suppose $p>1/2$, $\lambda > \lambda_c(\ZZ)$, and $\delta>0$ and consider the nearest neighbor contact process on $Q^2_N$ with bond dilution in which edges were kept with probability $p$. Starting from all sites occupied, the process survives for time
$$
\sigma_N \ge \exp(\gamma_2(\lambda)\eta_2(p) N/\log N)
\qquad\hbox{when $N$ is large.}
$$ 
\end{theorem}

\noindent
This falls short of the gold standard of survival for time $\exp(c(\lambda,p)N)$ but not by much. Of course if one can prove that there is an $\eta(p)>0$ so thatpaths of length $\ge \eta(p)N$ exist for $p>p^*$ then exponential survivial time follows. This should also be true on tori in $d\ge 3$, and this might not be too hard for an expert in percolation.

\clearp

\section{Griffiths' phase}

\subsection{\bf Results in one dimension}

Noest \cite{Noe88} studied oriented site percolation on 
$$
{\cal L}^1 = \{ (m,n) : m \in \ZZ, n \ge 0, m+n \hbox{ is even}\}.
$$ In this system if $(x-1,n)$ or $(x+1,n)$ is occupied then with probability $\theta$, $(x,n+1)$ is occupied, otherwise $(x,n+1)$ is vacant. Here we use the words result and derivation rather than theorem and proof to indicate that the result we are about to state is not a rigorous result. The 3 here means that it is the physics analogue of Theorem 3.

\mn
{\bf Physics Result 3.}
{\it Suppose that in site dilution on $[1,N]$ the fraction of sites retained is $p<1$.  If the process starts from all sites occupied, then there is a constant $c(p,\theta)$ so that the fraction of occupied sites at a large time $t$ satisfies
\beq
u(t) \ge c  (at/b)^{-b/a} \log(at/b) \quad\hbox{as $t\to\infty$,}
\label{Noestasy}
\eeq
where $a=\gamma_2(\theta)$, and $b = \log(1/p)$. }

\mn
{\it Derivation.} Here we follow page 2717 of Noest \cite{Noe88}. 
The probability that a string of length $n$ of active sites starting from all sites occupied has not reached the all 0's state at time $t$
$$
u_n(t) = P( \sigma_n > t ) = \exp(-t/T_n) \quad\hbox{where}\quad 
T_n \approx \exp(a(\theta) n).
$$
The probability of occurrence of an interval of exactly $n$ good sites is $P_n=p^n(1-p)^2$.
The fraction of occupied sites 
$u(t) = \sum_n n P_nu_n(t)$. 
Thus if $A=a(\theta)$ then the ``effective decay time'' 
\begin{align*}
T=\sum _t u(t) &=\sum_n nP_n \cdot ET_n \\
&\sim (1-p)^2 \sum_{n=1}^\infty  (Ap)^n  = Ap(1-p)^2(1-Ap)^{-2},
\end{align*}
if $p<1/A$. The asymptotic behavior of $u(t)$ can be obtained from 
\begin{align}
u(t) & = \sum_n n p^n \exp(-t \exp(-an) )
\nonumber \\
& \approx \int_0^\infty dx\,  x \exp[-bx - t\exp(-ax)],
\label{Neq4}
\end{align}
where $b=-\log p$. To prepare to use Laplace's method. 
Let 
\begin{align*}
\phi(x) & = -bx - t \exp(-ax) ,\\ 
\phi'(x) & = -b +at \exp(-ax) = 0 \quad\hbox{when $x^*=(1/a) \log(at/b)$}, \\
\phi''(x) & = -a^2t\exp(-ax) \quad \hbox{so $\phi''(x^*) = -ab$}.
\end{align*}
At the maximum 
$$
\phi(x^*) = -(b/a) \log(at/b)  -t \exp(-\log(at/b)) = -(b/a) [\log(at/b) +1],
$$
which means $\exp(\phi(x^*)) = (at/b)^{-b/a} \cdot e^{-b/a}$.
\eopt

\medskip
To be able to compare with the conclusion of the next result, note that if we forget about the log factor and the constants then $u(t) = 1/N$ (and hence there are $O(1)$ occupied sites) when 
\beq
t = N^{a/b} = N^{\gamma_2(\theta)/\log(1/p)}.
\label{translate}
\eeq
Turning to our rigorous result.

\begin{theorem} \label{Gphase}
Suppose $p<1$, $\lambda>\lambda_c(\ZZ)$, and $\delta>0$. The site diluted contact process on $[1,N]$ starting from all sites occupied survives for time 
\beq
\sigma_N \ge N^{(1-\delta) \gamma_2(\lambda)/\log(1/p)}
\quad\hbox{for large $N$}.
\label{lb1D}
\eeq
\end{theorem}

\noindent
The key is the fact  show that the largest interval of active sites in $[1,N]$,
$$
L(p) \sim \log N/\log (1/p).
$$ 
see Lemma \ref{maxL1D}.
Then we use the result for the survivial of the one dimensional contact process given in Lemma \ref{d1surv}.

\subsection{\ER graphs}

 We begin by giving the argument from Section II.C of Juh\'asz et al \cite{JOCM}.The ingredients in the proof are the following. 

\begin{itemize}

\item
A well-known result  for \ER graphs which can be found in many places  (see e.g., (1.6.10) in \cite{DoG}) implies that when $\nu \le 1$ the network of active nodes consists of finite clusters whose distribution is given by
\beq
P(s) \sim \frac{1}{\nu \sqrt{2\pi}} s^{-3/2} e^{-s\alpha(\nu)}
\quad\hbox{where $\alpha(\nu) = \nu-1-\ln(\nu)$}
\label{Pofs}
\eeq
and $\nu = \mu p$ is the average degree in the diluted graph.

\item
As in the previous example, the long-time decay of the fraction of occupied sites $u(t)$ can be written as the following integral
\beq
u(t) \sim \int ds \, s P(s) \exp(-t/\tau(s)),
\label{rhot}
\eeq
where $\tau(s)$ is the characteristic decay time of a region of size $s$

\item
$\tau(s)$ grows exponentially (Arrhenius law) with the cluster size 
\beq
\tau(s) = \exp(A(\lambda)s),
\label{taus} 
\eeq
where  $A(\lambda)$ does not depend on $s$.

\end{itemize}

\noindent
Using a saddle-point approximation with \eqref{rhot}, see Section \ref{sec:PfER} for details, one obtains

\mn
{\bf Physics Result 4.}
{\it Suppose $\nu = \mu p < 1$,  $\lambda > \lambda_c(\ZZ)$, and $\delta>0$.  In the edge diluted contact process on \ERns$(N,\mu/N)$ the fraction of occupied sites
 satisifes}
$$
u(t) \sim t^{-\theta(\nu,\lambda)}\quad\hbox{where}\quad
\theta(\nu,\lambda) = -\alpha(\nu)/A(\lambda).
$$

\mn
Converting the decay rate to survival time as before, $u(t)=1/N$ when 
\beq
t=N^{A(\lambda)/a(\nu)}
\label{physpre}
\eeq

The claim that the survival times only depends on the size of the cluster
is too good to be true, but it is also more than we need. We leave as an open problem to

\mn
{\bf Prove or disprove.} There are constants $c,C$ that depend on the graph but are independent of $\lambda$ and $s$  so that
$$
\exp(cA(\lambda) s ) \le \tau(s) \le \exp(CA(\lambda) s ) \quad \hbox{for large $s$}.
$$

\medskip
Our rigorous result is

\begin{theorem} \label{GphaseER}
Suppose $\nu = \mu p < 1$,  $\lambda > \lambda_c(\ZZ)$, and $\delta>0$.  The  contact process on bond diluted \ERns$(N,\mu/N)$ starting from all sites occupied survives for time
\beq
\sigma_N \ge N^{(1-\delta) \gamma_2(\lambda)/\log(1/\nu)}
\qquad\hbox{for large $N$.}
\label{ERsurvbd}
\eeq
\end{theorem}

\mn
Again the key result is that the longest path,  $\sim (\log N)/\log(1/\nu)$, see  Section \ref{sec:PfER} for details. Then we use the result for the survivial time of the contact process given in Lemma \ref{d1surv}.

\medskip
To compare the tlast wo results we note that  Corollary 5.11 of Bollob\'as \cite{Boll} or Theorem 1.3.1. in \\cite{DoG})shows that the largest cluster in graph with $N$ vertices
$$
= \frac{1}{\alpha(\nu)}\left( \log N - \left( \frac{5}{2} + o(1) \right) \log\log N \right).
$$ 
where $\alpha(\nu)=\nu-1-\log(\nu)$ is the constant in \eqref{Pofs}, while in the rigorous resulty we use the longest path, which has length 
$$
\sim\log N/ \log(1/\nu). 
$$
To compare the two sizes we note that if $\nu = 1 - \ep$ then  
$$
\log(1/\nu) = - \log(\nu) =  \ep + \frac{\ep^2}{2} +\frac{\ep^3}{3} \ldots
$$
so $\alpha(\nu) =\log(1-\ep) +\ep \sim \ep^2/2$. Thus the large cluster is of size $\sim 2\ep^{-2}\log N$ while the longest path is of size $\sim \ep^{-1} \log N$

\subsection{Two dimensions}

As with one dimensional systems and \ER graphs, we can establish the long time persistence in the Griffiths phase by showing the existence of paths of length $O(\log N)$. Thanks to a result of Grimmett \cite{Gri81} given in Lemma \ref{Gr1981} this is easy. He proves that there is a positive constant $\eta_2(p)$ so that paths of length $\ge \eta_2(p) \log N$ exist in subcritical two-dimensional percolation.

\begin{theorem} \label{Gphase2D}
Suppose $p<1/2$, $\lambda > \lambda_c(\ZZ)$, and $\delta>0$.  The nearest neighbor contact process with randomly deleted edges on $\ZZ^2 \cap [1,N^{1/2}]^2$  starting from all sites occupied survives for time
\beq
\sigma_N \ge N^{(1-\delta) \eta_2(p) \gamma_2(\lambda)}
\qquad\hbox{for large $N$.}
\label{d2Gphase}
\eeq
\end{theorem}

\subsection{On the critical line}

There is no critical line in $D=1$. On \ER graphs and in two dimensions we will show that the longest path is $c N^\alpha$, so using or result for the one dimensional contact process, Lemma \ref{d1surv}, the system survives for time
$\ge \exp(c \gamma_2(\lambda) n^\alpha)$. 
Computing as we have several times before if $u(t) = 1/(\log t)^{1/\alpha}$ then $u(t) = 1/n$ when 
$$
t = \exp( N^{\alpha}).
$$

\mn
{\bf Proof for the \ER graph.} We claim that at criticality the longest path will be $\ge cN^{1/3})$. To argue this informally, it is known that the largest cluster at criticality has $\Theta(N^{2/3})$ vertices. Critical clusters are like critical branching processes.  A critical branching process that survives for time $T$ has $\Theta(T^2)$ individuals, so skipping more than a few steps the longest path in a cluster of size $\Theta(N^{2/3})$ should be $\Theta(N^{1/3})$. For a rigorous proof see Addario-Berry, Broutin, and Goldschmidt \cite{Louigi1,Louigi2} who also show  that critical \ER clusters rescaled by $N^{-1/3}$ converge to a sequence of compact metric spaces,  \eopt 

\mn
{\bf Proof for two dimensions.} At criticality crossings of an $L \times L$ box have probability $\approx 1/2$, so taking $L=N^{1/2}$ we see that the longest path has length $\ge c N^{1/2}$ with probability $\ge 1/2$.. To argue that we have a path of this length with high probability, we divide the $L \times L$ square into $k^2$, $L/k \times L/k$ squares and note that the probability they all fail to have crossings is $\exp(-(\log 2) k^2)$. \eopt

\clearp

\section{Proofs in one dimension} \label{sec:Pf1D}

\begin{proof} [Proof of Theorem \ref{Gphase}.]
Inert sites cannot give birth, so the contact processes on disjoint (maximal) intervals of active sites are independent. Suppose that the site 1 is inert. The interval $[1,N]$ begins with an inert interval of length $B_1$ that has
$$
P(B_1=k) = (1-p)^{k-1} p \quad\hbox{ for $k \ge 1$.}
$$
followed by an active interval with length $A_1$ with
$$
P(A_1=k) =p^{k-1}(1-p)  \quad\hbox{ for $k \ge 1$.}
$$
so $EB_1=1/p$ and $EA_i=1/(1-p)$. We can repeat these definitions until the interval $[1,N]$ is used up. A cycle $B_i,A_i$ has expected length 
$$
EB_i + EA_i = \frac{1}{p} + \frac{1}{1-p} = \frac{1}{p(1-p)}.
$$
As $N\to\infty$ the number of active intervals, $M(p)$,  in $[1,N]$ has
$$
M(p) \sim Np(1-p).
$$

\begin{lemma} \label{maxL1D}
Suppose that $A_i$, $i \ge 1$ are independent geometric($1-p$) and let
$L(p) = \max_{1\le i \le M(p)} A_i$ Then 
$$
L(p)/\log N \to 1/\log(1/p).
$$
\end{lemma}

\begin{proof}
To prove this we note that $P(A_i > K ) = p^K$. Let $L_m = \max_{1\le i \le m} A_i$
\begin{align*}
P( L_m > (1+\delta)(\log m) /\log(1/p) ) 
\le m p^{ (1+\delta)(\log m) /\log(1/p) } &\\
 = m \exp( \log(p) (1+\delta)(\log m) /\log(1/p) ) & = m^{-\delta} \to 0.
\end{align*}
Let $Y_{p,\delta}$ be the expected number of $A_i$ with $1\le i \le m$ and
$A_i > (1-\delta)(\log m) /\log(1/p)$. Repeating the last caclulation with $(1+\delta)$ replaced by $(1-\delta)$ we have $EY_{p,\delta} \sim m^\delta$. Since $Y_{p,\delta}$ is binomial with a small success probability, the variance is also $\sim m^\delta$, so Chebyshev's inequality implies
$$
P( Y_{p,\delta}\le m^\delta/2) \le \frac{m^{\delta}}{m^{2\delta}/4} \to 0.
$$ 
This shows that $L_m/\log m \to 1/\log(1/p)$.
Plugging in $m=Np(1-p)$ provs Lemma \ref{maxL1D}.
\end{proof}

\noindent
Using Lemma \ref{d1surv} we see that for any $\delta>0$ the contact process on the longest active interval in $[1,N]$ survives for time.
$$
\ge \exp( (1-\delta) \gamma_2(\lambda) \log N /\log(1/p)),
$$
which proves Theorem \ref{Gphase}. \end{proof}

\clearp

\section{Proofs for \ER graphs} \label{sec:PfER}

\subsection{Completion of the derivation of Physics Result 4}

Plugging \eqref{taus} and \eqref{Pofs} into \eqref{rhot} gives
$$
u(t) \sim \int ds \, s \frac{1}{\sqrt{2\pi} p} s^{-3/2} e^{-s\alpha(\nu)} 
\exp(-t \exp(-A(\lambda) s ) ,
$$
where $\alpha(\nu) = \log(\nu) - 1- \nu$.
To maximize the integrand (ignoring the $s^{-3/2}$) we  look only at what is inside the exponential and let $M(s) =  -s\alpha(\nu) -t \exp(-A(\lambda) s )$. Taking a derivative
$$
\frac{d}{ds}M(s) 
= - \alpha(\nu) + t A(\lambda) \exp(-A(\lambda) s ).
$$
This is $=0$ at $s_0=0$ when
$$
\exp(-A(\lambda) s_0 ) =  \frac{\alpha(\nu)}{t A(\lambda)}
\quad\hbox{or}\quad s_0 = \frac{\log t + O(1)}{A(\lambda)}.
$$
Taking $s=s_0$ and dropping the $O(1)$ 
$$
M(s_0) = -\frac{\alpha(\nu)(\log t)}{A(\lambda)} -t \frac{\alpha(\nu)}{t A(\lambda)}.
$$
The second term is $O(1)$ so the maximum value of the integrand is indeed $t^{-\theta(\nu,\lambda)}$with $\theta(\nu,\lambda) = \alpha(\nu)/A(\lambda)$.
Detailed computation of the asymptotics of the integral are left to the reader \eopt

\subsection{Proof of Theorem \ref{GphaseER}.} 

The first step is to prove

\begin{lemma} \label{maxpER}
If $\nu<1$ then as $N \to\infty$ the longest path in an \ERns$(N,\nu/N)$
$\sim (\log N)/\log(1/\nu)$.
\end{lemma}

\begin{proof}
If $k^2/N \to 0$, which we assume throughout the proof, then the expected number of (self-avoiding) paths of length $k$ in an Erd\"os-R\'eny($N,\nu/N$) graph
$$
\Pi(N,k,\nu) = N(N-1) \cdots (N-k+1) \cdot (\nu/N)^{k-1} \sim N \nu^{k-1} .
$$
$\Pi(N,k_1,\nu) \approx 1$ when $k_1(\nu,N)= (\log N)/\log (1/\nu)$. If $\bar k = (1+\delta)k_1$ then 
$$
\Pi(N,\bar k,\nu) \le N^{-\delta} \to 0 ,
$$
which gives the upper bound on the length of the longest path. 

If $\underline{k} = (1-\delta) k_1$ then $\Pi(N,\underline{k},\nu) \sim N^\delta$. To prove the lower bound let $\pi$ be a sequence of distinct vertices of length $\underline{k}$, let $A_\pi$ be the event that $\pi$ is a paths of length $\underline{k}$ in the graph, and let $Y_\delta$ be the number of path of length $\underline{k}$. If $\pi$ and $\sigma$ do not share an edge in common then $A_{\pi}$ and $A_\sigma$ are independent. Let $\Sigma^j$ be the number of pairs of paths $\pi$ and $\sigma$ that have exactly $j$ edges in common. 

We get a lower bound on $\Sigma^0$ by assuming all vertices in each path are different.
$$
N(N-1) \cdots (N-2k-1)  (\nu/n)^{2k-2} \le \Sigma^0  \le \Pi^2 ,
$$
so $\Sigma^0 \sim \Pi^2= (N \nu^{k-1})^2$. In $\Sigma^1$ we need to pick the edge to be the same.
\begin{align*}
\Sigma^1  \le N^k \cdot k N^{k-2} \cdot ( \nu/N)^{(k-1)+(k-2)}& \\
\sim kN\nu^{2k-3}= (\Pi)^2 \cdot k \nu^{-1}/n = o(\Sigma^0). &
\end{align*}
 When it comes to $\Sigma^2$ the two edges can be adjacent in the path or not
\begin{align*}
\Sigma^2 \approx N^k \cdot [k N^{k-3} + k^2 N^{k-4}] 
\cdot ( \nu/N^{(k-1)+(k-3)})& \\
\le (\Pi)^2 \cdot \nu^{-2}  [k/N+ k^2/N^2] = o(\Sigma^0)&
\end{align*}
The number of possibilities increases as the number of duplicates increases but the best case occurs when all the agreements are in a row. Since $\sigma^0 \sim \Pi^2$ the square of the mean, it follows that the variance is $o(\Pi^2)$, and the ratio $Y_\delta/EY_\delta \to 1$.
\end{proof}

\clearp

\section{Proofs in two dimensions} \label{sec:Pf2d}

\subsection{Supercritical phase}

Let $\LL^2$ be the graph with vertex set $\ZZ^2$ and edges connecting nearest neighbors. We begin by describing planar graph duality for $\LL^2$.  Each edge in $\LL^2$ is associated with the edge in $(1/2,1/2) + \LL^2$ that intersects it, see Figure 3. We begin by making the edges in $\LL^2$ independently open with probability $p$, and then declaring an edge in $(1/2,1/2) + \LL^2$ to be open if and only if it is paired with a closed edge.

\begin{figure}[ht]
\begin{center}
\begin{picture}(280,240)
\put(40,20){\line(1,0){160}}
\put(40,60){\line(1,0){160}}
\put(40,100){\line(1,0){160}}
\put(40,140){\line(1,0){160}}
\put(40,180){\line(1,0){160}}
\put(40,220){\line(1,0){160}}
\put(40,20){\line(0,1){200}}
\put(80,20){\line(0,1){200}}
\put(120,20){\line(0,1){200}}
\put(160,20){\line(0,1){200}}
\put(200,20){\line(0,1){200}}
\thicklines
\linethickness{0.75mm}
\put(20,40){\line(1,0){200}}
\put(20,80){\line(1,0){200}}
\put(20,120){\line(1,0){200}}
\put(20,160){\line(1,0){200}}
\put(20,200){\line(1,0){200}}
\put(20,40){\line(0,1){160}}
\put(60,40){\line(0,1){160}}
\put(100,40){\line(0,1){160}}
\put(140,40){\line(0,1){160}}
\put(180,40){\line(0,1){160}}
\put(220,40){\line(0,1){160}}
\put(16,20){\bf 0}
\put(216,20){\bf 5}
\put(7,35){\bf 0}
\put(7,195){\bf 4}
\end{picture}
\caption{Planar graph duality. If there is a left-to-right crossing of $[0,n+1] \times [0,n]$ if and only if there is no top-to-bottom crossing of $(1/2,-1/2) + [0,n] \times [0, n+1]$.  When bonds are open with probability 1/2 the two crossings have the same probability and add up to 1, so they are both $=1/2$.}
\end{center}
\label{fig:2dpgd}
\end{figure}

The picture of the duality drawn in Figure 3 led to Ted Harris' proof in \cite{TEH60} that the critical value for two dimensional bond percolation is $\ge 1/2$. This was the starting point for developments that led to Kesten's proof \cite{Kes80} that the critical probability was $1/2$. Computing the critical value required the development of a machinery for estimating sponge crossing probabilities. See Chapters 4, 6, and 7 in Kesten \cite{Kes82}, or Chapter 11 in Grimmett \cite{Gri99}). We will avoid this machinery by using the fact that in the subcritical regime there is exponential bound on the radius of clusters.

\begin{lemma} \label{crossprob}
Consider bond percolation with $p>1/2$. There is a constant $\gamma(p)>0$ so that the  probability of a left-to-right crossing of $[0,L] \times [0,K]$ is 
$$
\ge 1 - L\exp(-\gamma(p) K).
$$
\end{lemma}

\noindent
From Lemma \ref{crossprob} follows that if $K = C\log L$ and $C > 3/\gamma$ then the probability of a left-to-right crossing of $[0,L] \times [0,C\log L] \ge 1 - L^{-3}$.
For $k \le K = \lfloor L /\log L \rfloor$ define boxes
\begin{align*}
H_k & = [0,L] \times [(k-1)C\log L,kC\log L] \\
V_k^0 & = [0,C\log L,L] \times [(k-1)C\log L,(k+1) C\log L] \\
V_k^L & = [L -C\log L,L] \times [(k-1)C\log L,(k+1) C\log L]
\end{align*}

\begin{figure}[ht]
\begin{center}
\begin{picture}(240,160)
\put(15,5){0}
\put(215,5){$L$} 
\put(225,35){$C \log L$}
\put(225,55){$2C \log L$}
\put(225,75){$3C \log L$}
\put(225,95){$4C \log L$}
\put(225,115){$5C \log L$}
\put(20,32){\line(1,0){200}}
\put(20,47){\line(1,0){200}}
\put(20,74){\line(1,0){200}}
\put(20,91){\line(1,0){200}}
\put(20,113){\line(1,0){100}}
\put(122,113){$\ldots$}
\put(205,20){\line(0,1){40}}
\put(212,60){\line(0,1){40}}
\put(28,40){\line(0,1){40}}
\put(34,80){\line(0,1){40}}
\put(120,125){$\vdots$}
\thicklines
\linethickness{0.75mm}
\put(20,20){\line(1,0){200}}
\put(20,40){\line(1,0){200}}
\put(20,60){\line(1,0){200}}
\put(20,80){\line(1,0){200}}
\put(20,100){\line(1,0){200}}
\put(20,120){\line(1,0){200}}
\put(20,20){\line(0,1){140}}
\put(220,20){\line(0,1){140}}
\put(200,20){\line(0,1){100}}
\put(40,40){\line(0,1){80}}
\end{picture}
\caption{Picture of the construction of a long path.}
\end{center}
\label{fig:2dpgd}
\end{figure}

 Given this result if we let $L=N^{1/2}$ then we can create a path of length $ \ge N/\log N$ in $[0,L]^2$ by combining

\medskip
left-to-right crossings of $H_k$  for odd $k \le K$

right-to-left crossings of $H_k$ for even $k\le K$

bottom-to-top crossings of  $V^L_k$ for odd $k \le K$

bottom to top crossings of  $V^0_k$ for even $k\le K$

\mn
The left to right crossings have length $\ge L$ while the number of horizontal strips is
$\ge \lfloor N/C \log N \rfloor$, so we have a path of length $\ge \kappa_2(p)N/\log N$. The existence of such a path in combination with Lemma \ref{d1surv} gives
Theorem \ref{dens2D}

\subsection{Griffiths phase} 

As in one dimension and for  \ER graph we establish the long time persistence in the Griffiths phase by showing the existence of long paths. Thanks to a result of Grimmett \cite{Gri81} this is easy

\begin{lemma} \label{Gr1981}
Consider bond percolation on the square lattice with $p<1/2$. Let $S(L)$ be the probability that some open path joins the longer sides of a sponge with dimensions $L$ by $a \log L$.  There is a positive constant $\alpha$ which depends on $p$ so that as $L\to\infty$
$$
S(L) \to \begin{cases} 0 & \hbox{if $a\alpha>1$}\\ 
1 & \hbox{if $a\alpha<1$}
\end{cases}
$$
\end{lemma}

\noindent
Let $L=N^{1/2}$. This implies the existence of paths of length $\eta_2(p) \log L$ where $\eta_2(p) =  1/2\alpha(p)$. Combining this result with Lemma \ref{d1surv} proves Theorem \ref{Gphase2D}.

\clearp

\end{document}